\documentclass[letterpaper]{amsart}
\usepackage{amssymb}
\usepackage{amsthm}
\usepackage{amsmath}
\usepackage{amsfonts}
\usepackage{comment}
\usepackage[mathscr]{euscript}
\usepackage{tipa}
\usepackage{url}
\usepackage{graphicx}
\usepackage{xfrac}
\usepackage{MnSymbol}
\usepackage{enumitem}
    
\usepackage{soul}

\usepackage{xargs}

\newcommandx{\change}[2][1=]{\todo[linecolor=blue,backgroundcolor=blue!25,bordercolor=blue,#1]{#2}}

\newcommand{\qq}{\mathbb{Q}}
\newcommand{\bbT}{\mathbb{T}}
\newcommand{\rr}{\mathbb{R}}

\newcommand{\nn}{\mathbb{N}}

\newcommand{\zz}{\mathbb{Z}}

\newcommand{\la}{\langle}
\newcommand{\ra}{\rangle}

\newcommand{\sM}{\mathscr{M}}

\makeatletter
\newsavebox\myboxA
\newsavebox\myboxB
\newlength\mylenA

\newcommand*\xbar[2][0.75]{%
    \sbox{\myboxA}{$\m@th#2$}%
    \setbox\myboxB\null
    \ht\myboxB=\ht\myboxA%
    \dp\myboxB=\dp\myboxA%
    \wd\myboxB=#1\wd\myboxA
    \sbox\myboxB{$\m@th\overline{\copy\myboxB}$}
    \setlength\mylenA{\the\wd\myboxA}
    \addtolength\mylenA{-\the\wd\myboxB}%
    \ifdim\wd\myboxB<\wd\myboxA%
       \rlap{\hskip 0.5\mylenA\usebox\myboxB}{\usebox\myboxA}%
    \else
        \hskip -0.5\mylenA\rlap{\usebox\myboxA}{\hskip 0.5\mylenA\usebox\myboxB}%
    \fi}
\makeatother

\newtheorem{thm}{Theorem}[section]
\newtheorem{lem}[thm]{Lemma}

\newtheorem*{thm*}{Theorem}
\newtheorem*{conj*}{Conjecture}
\newtheorem*{prop*}{Proposition}
\newtheorem*{fact*}{Fact}
\newtheorem{prop}[thm]{Proposition}

\newtheorem{cor}[thm]{Corollary}

\newtheorem*{eg*}{Example}

\usepackage[usenames, dvipsnames]{color}
 
\definecolor{red}{rgb}{1.0, 0, 0}

\begin{document}
\sloppy


\title[ A family of dp-minimal expansions of $(\zz; +)$]{A family of dp-minimal expansions of $(\zz;+)$}
\thanks{This is a preprint. Later versions might contain significant changes.  Comments
are  welcome!}
\author{Minh Chieu Tran, Erik Walsberg}
\address{Department of Mathematics, University of Illinois at Urbana-
Champaign, Urbana, IL 61801, U.S.A}
\curraddr{}
\email{mctran2@illinois.edu, erikw@illinois.edu}
\date{\today}

\begin{abstract}
We show that the cyclically ordered-abelian groups expanding $(\mathbb{Z};+)$ contain a continuum-size family of dp-minimal structures such that no two members define the same subsets of $\zz$.
\end{abstract}

\maketitle
\section{Introduction}

\noindent In this paper, we are concerned with the following classification-type question:
\begin{center}
\textit{What are the dp-minimal expansions of $(\mathbb{Z};+)$?}
\end{center}
For a definition of dp-minimality, see \cite[Chapter 4]{simon-book}. The terms expansion and reduct here are as used in the sense of definability: If $\sM_1$ and $\sM_2$ are structures with underlying set $M$ and every $\sM_1$-definable set is also definable in $\sM_2$, we say that $\sM_1$ is a {\it reduct} of $\sM_2$ and that $\sM_2$ is an {\it expansion} of $\sM_1$. Two structures are {\it definably equivalent} if each is a reduct of the other.

\medskip \noindent A very remarkable common feature of the known dp-minimal expansions of $(\mathbb{Z};+)$ is their ``rigidity''. In \cite{CoPi}, it is shown that all proper stable expansions of $(\mathbb{Z};+)$ have infinite weight, hence infinite dp-rank, and so in particular are not dp-minimal.
The expansion $(\mathbb{Z};+,<)$, well-known to be dp-minimal, does not have any proper dp-minimal expansion \cite[6.6]{toomanyI}, or any proper expansion of finite dp-rank, or even any proper strong expansion \cite[2.20]{DoGo}. Moreover, any reduct $(\mathbb{Z};+,<)$ expanding $(\zz;+)$ is definably equivalent to $(\zz;+)$ or $(\zz;+,<)$ \cite{conant}. Recently, it is shown in  \cite[1.2]{AldE} that $(\mathbb{Z}; +, \prec_p)$ is dp-minimal for all primes $p$ where  $\prec_p$ be the partial order on $\mathbb{Z}$ given by declaring $k \prec_p l$ if and only if $v_p(k) < v_p(l)$ with $v_p$ the $p$-adic valuation on $\mathbb{Z}$. Also, any reduct of $(\mathbb{Z}; +, \prec_p)$ expanding $(\zz;+)$ is definably equivalent to either $(\mathbb{Z};+)$ or $(\zz;+,\prec_p)$ \cite[1.13]{AldE}.

\medskip \noindent The above ``rigidity'' gives hope for a classification of dp-minimal expansions of $(\mathbb{Z};+)$ analogous to that of dp-minimal fields \cite{Johnson}. 
In  \cite[5.32]{toomanyII}, the authors asked whether every dp-minimal expansion of $(\mathbb{Z};+)$ is a reduct of $(\mathbb{Z};+,<)$. In view of  \cite[1.2]{AldE}, the natural modified question is whether every dp-minimal expansion of $(\mathbb{Z};+)$ is a reduct of $(\mathbb{Z};+,<)$ or $(\mathbb{Z}; +, \prec_p)$ for some prime $p$.

\medskip \noindent In this paper,  we give a strong negative answer to the above question. We introduce cyclically ordered-abelian groups expanding $(\mathbb{Z};+)$ in Section 2, show that these are all dp-minimal, and all except two are not reducts of known examples. In Section 3, we characterize unary definable sets in these expansions of $(\mathbb{Z};+)$, classify these structures up to definable equivalence, and show that there are continuumm many up to definable equivalence. The proof of many of the above results notably makes use of Kronecker's approximation theorem.


\subsection*{Notations and conventions} Throughout, $j$, $k$, and $l$ range over the set $\zz$ of integers, $m$ and $n$ range over the set $\nn$ of natural numbers (which includes 0). The operation $+$ on $\nn$, $\zz$, $\zz^2$, $\qq$, and $\rr$ are assumed to be the standard ones and likewise for $\times$ and the ordering $<$ on all the above except $\zz^2$. If $(M; <)$ is a linear order and $a, b \in M$, set $[a, b)_M = \{ t \in M : a \leq t <b\}$, likewise for other intervals.

\section{Cyclically ordered abelian groups expanding $(\zz; +)$}
\noindent We begin with some general facts on cyclically ordered groups.
There are many sources for this material, \cite[8.1]{Gu-thesis} is one.
A \textbf{cyclic order} on a set $G$ is a ternary relation $C \subseteq G^3$ such that for all $a,b,c \in G$, the following holds:
\begin{enumerate}
\item if $(a,b,c) \in C$, then $(b,c,a) \in C$;
\item if $(a,b,c) \in C$, then $(c,b,a) \notin C$;
\item if $(a,b,c) \in C$ and $(a,c,d) \in C$ then $(a,b,d) \in C$;
\item if $a,b,c$ are distinct, then either $(a,b,c) \in C$ or $(c,b,a) \in C$.
\end{enumerate}
We will often write $C(a, b, c)$ instead of $(a, b, c) \in C$.
It is easy to see that the binary relation $<$ on $G$ given by declaring $a < b$ if either $C(0,a,b)$ or $a = 0$ and $b \neq 0$ is a linear order, which we will refer to as the linear order on $G$ {\bf associated} to $C$.

\medskip \noindent A cyclic order $C$ on the underlying set of an abelian group  $(G;+)$ is \textbf{additive} if it is preserved under the group operation. In this case, we call the combined structure  $(G; +, C)$ a \textbf{cyclically ordered abelian group}. 

\medskip \noindent Let $(G; +)$ be an abelian group. Suppose $(H; +,<)$ is a linearly ordered abelian group, $u$ is an element in $H^{>0}$ such that $( nu)_{n>0}$ is cofinal in $(H; <)$, and $\pi: H \to G$ induces an isomorphism from $(H \slash \la u \ra; +)$ to $(G; +)$. Define the relation $C$ on $G$ by:  
$$ C(\pi(a), \pi(b), \pi(c)) \ \text{ if }\  a<b<c \text{ or } b< c< a \text{ or } c< a < b \quad \text{ for }a, b, c \in [0,u)_H.$$
We can easily check that $C$ is an additive cyclic ordering on  $(G;+)$. We call $(H; u, +,<)$ as above a {\bf universal cover} of $(G;+,C)$ and $\pi$ a {\bf covering  map}. 
It turns out that all cyclically ordered abelian groups can be obtained in this fashion:
\begin{lem}\label{prop:cover}
Suppose $(G; +, C)$ is a cyclically ordered abelian group. Then $(G; +, C)$  has a universal cover  $(H; u, +,<)$   which is unique up to unique isomorphism.  Moreover, $(G; +, C)$ is isomorphic to $([0, u)_H; \tilde{+},\tilde{C})$ where $\tilde{+}$ and $\tilde{C}$ are definable in $(H; u, +,<)$. 
\end{lem}
\begin{proof}
We set $H = \zz \times G$ and for $(k, a)$ and $(k', a')$ in $H$, define 
$$(k, a)+ (k', a') = 
\begin{cases}
(k+k', a+a' ) & \text{ if } a =0  \text{ or } a' = 0 \text{ or } C(0_G, a, a +a') ,\\
(k+k'+1, a+a' ) & \text{ otherwise}.
\end{cases}
 $$
We let $<$ be the lexicographic product of the usual order on $\zz$ and the linear order on $G$ associated to $C$. Set $0_H = (0_\zz, 0_G)$ and $u = (1_\zz, 0_G)$. We can easily check that $(H;u,+,<)$ is a universal cover of $G$ and that every universal cover of $(G; +, C)$  is isomorphic to $(H; u, +,<)$.
For $a,a' \in [0_H,u)_H$, we set
$$ a \tilde{+} a'=
\begin{cases}
a + a' & \text{ if } a + a' \in [0_H,u)_H ,\\
a + a' - u & \text{ otherwise}.
\end{cases}
$$
We define $\tilde{C}$ by setting $\tilde{C}(a,b,c)$ for any $a,b,c \in [0,u)_H$ such that $a < b < c$ or $b < c < a$ or $c < a < b$.
It is easy to see the quotient map $H \to G$ induces an isomorphism $([0_H,u)_H,\tilde{+}, \tilde{C}) \to (G,+,C)$.
\end{proof}

\noindent Lemma~\ref{prop:cover} gives us a correspondence  between additive cyclic orderings on $\zz$ and  additive linear orderings on $\zz^2$:
\begin{prop} \label{prop: cyclicvslinear}
Let $(\zz; +, C)$ be a cyclically ordered group. Then there is a linear order $<$ on $\zz^2$ such that a universal cover of $(\zz; +, C)$ is $(\zz^2; u, +, <)$ with $u =(1,0)$.
\end{prop}
\begin{proof}
Suppose $(\zz; +, C)$ is as above and $(H; u, +, <)$ is its universal cover. Then $(\zz; +)$ is $(H\slash \la u \ra; +)$. Using also the fact that $(\la u \ra; +)$ is isomorphic to $(\zz; +)$, we arrange that $(H; +)$ is $(\zz^2; +)$. Choose $v \in \zz^2$ such that $v$ is mapped to $1$ in $\zz$ under the quotient map. Then $\la u, v\ra = \zz^2$, and so by a change of basis we can arrange that $u =(1, 0)$.
\end{proof}

\noindent The dp-minimality of the  cyclically ordered groups $( \zz; +, C)$  can be established rather quickly using a criterion in \cite{JaSiWa}:

\begin{thm} \label{Thm: dp-minimality}
Every cyclically ordered group $( \zz; +, C)$ is dp-minimal.
\end{thm}
\begin{proof}
By the last statement of Lemma \ref{prop:cover} and Proposition \ref{prop: cyclicvslinear}, it suffices to check that every linearly ordered group $( \zz^2; +, <)$ is dp-minimal. We have that $$ |  \zz^2 \slash n \zz^2|  = n^2 < \infty.$$ The conclusion of the theorem follows from the criterion for dp-minimality in \cite[Proposition 5.1]{JaSiWa}.
\end{proof}

\noindent  So far it is still possible that every cyclically ordered group $(\zz; +, <)$ is a reduct of a known dp-minimal expansion of $(\zz; +)$. Toward showing that this is not the case, we need a more explicit description of the additive cyclic orders on $(\zz; +)$. 

\medskip \noindent Define the cyclic ordering $C_{+}$ on $\zz$ by setting $C_+( j, k, l)$ if and only if $j < k < l$ or $l < j < k$ or $k < l < j$.
We define the opposite cyclic ordering $C_{-}$ on $\zz$ by setting $$C_{-}(j,k,l) \text{ if and only if } C_{+}(-j,-k,-l).$$ We observe that $C_+$ and $C_-$ are distinct, but $(\zz; +, C_+)$ and $(\zz; +, C_-)$ are isomorphic via the map $k \mapsto -k$ and both have $(\zz^2; +, <_\text{lex}) $ as a universal cover where $<_\text{lex}$ is the usual lexicographic ordering on $\zz^2$. It is easy to see that both $(\zz;+,C_+)$ and $(\zz;+,C_{-})$ are definably equivalent with $(\zz;+,<)$.

\medskip \noindent Let  $( \rr \slash \zz; +, C)$ be the cyclically ordered group with a universal cover $(\rr; 1, +, <)$ and such that $C(0 + \zz, 1/4 + \zz, 1/2 + \zz)$ holds. We call $( \rr \slash \zz; +, C)$ the {\bf positively oriented circle}. For $a, b \in \rr$ such that $a-b \notin \zz$, we set $[a, b)_{\rr \slash \zz}$ to be the set 
$$ \{ t \in \rr \slash \zz: t =a +\zz \text{ or }  C(a+\zz, t, b +\zz)  \}.  $$
Let $\alpha$ be in $\rr \setminus \qq$. Define the additive cyclic ordering $C_\alpha$ on $(\zz; +)$ by   setting
$$ C_\alpha(j, k, l)\  \text{ if and only if }\   C(\alpha j + \zz, \alpha k +\zz , \alpha l +\zz). $$
In other words, $C_\alpha$ is the pull-back of $C$ by the character $\chi_\alpha : \zz \to \rr \slash \zz, l \mapsto \alpha l +\zz$. As before, we observe that $C_\alpha$ and $C_{-\alpha}$ are distinct. However, $(\zz; +, C_\alpha)$ and $(\zz; +, C_{-\alpha})$ are isomorphic via the map $k \mapsto -k$ and both have  $(\zz^2;+,<_\alpha)$ as a universal cover with $<_\alpha$ the pull-back of the ordering $<$ on $\rr$ by the group embedding
$$\psi_\alpha: \zz^2 \to \rr, \quad (k, l) \mapsto k+ \alpha l.$$
We also note that $(\zz^2; +, <_\alpha)$ is not isomorphic to  $(\zz^2; +, <_\text{lex})$ as the former is archimedean and the latter is not. It follows that $(\zz; +, C_\alpha)$ is not isomorphic to $(\zz; +,C_\text{+})$ and $(\zz;+,C_{-})$.

\noindent The following result is essentialy the well-known classification of linearly ordered group expanding $(\zz^2; +)$ up to isomorphism:
\begin{lem} \label{lem: linearorderclassification}
Suppose $(\zz^2; +, <)$ is a linearly ordered group such that $(nu)_{n>0}$ is cofinal in $\zz^2$ with  $u = (1, 0)$. Then $(\zz^2; u, +, <)$ is isomorphic to either  $(\zz^2; u, +, <_\mathrm{lex})$ or $(\zz^2; u, +, <_\alpha)$ for a unique $\alpha \in [0, 1/2)_{\rr \setminus \qq}$.
\end{lem}
\begin{proof}
Suppose $(\zz^2; +, <)$ and $u$ are as stated above. Using the fact that $(nu)_{n>0}$ is cofinal in $\zz^2$, we obtain $k$ such that $ku < (0, 1) < (k+1)u$.  Let   $v$ be $(0, 1) -ku$ if $2ku < (0, 2) < (2k+1)u$ and let $u$ be $(k+1)u -(0,1)$ otherwise. Then 
$$ \la u, v \ra = \zz^2 \text{ and } 0< 2v<u .$$
If $(nv)_{n>0}$ is not cofinal in $\zz^2$, then it is easy to see that the map 
$$\zz^2 \to \zz^2, \quad ku+lv \mapsto (k, l)$$
is an ordered group isomorphism from $(\zz^2; u, +, <)$ to $(\zz^2; u, +, <_\mathrm{lex})$. 
Now suppose $(nv)_{n>0}$ is cofinal in $\zz^2$. Then set
 $$ \alpha = \sup\left\{ \frac{m}{n} : m, n> 0 \text{ and } mu< nv  \right\}. $$
It is easy to check that $\alpha \in [0, 1\slash2)_{\rr\setminus \qq}$ and that  the map 
$\zz^2 \to \zz^2, ku+lv \mapsto (k, l)$
is an isomorphism from $(\zz^2; u,  +, <)$ to $(\zz^2; u,  +, <_\alpha)$. 

Finally, suppose  $\alpha$ and $\beta$ are in $[0, 1/2)_{\rr\setminus \qq}$ and $f$ is an isomorphism from 
$(\zz^2; u, +, <_\alpha)$ to $(\zz^2; u, +, <_\beta)$ with $u =(1, 0)$. Let $v = (0, 1)$.
Then  
$$  \la u, f(v) \ra = \zz^2 \text{ and }  0< 2f(v) < u .$$
The former condition implies $f(v)$ is either $(k, 1)$ or $(k, -1)$ for some $k$. Combining with the latter condition, we get $f(v) = (0, 1)$, and so $f = \text{id}_{\zz^2}$. 
 It follows easily from the definition of $<_\alpha$ and $<_\beta$ that $\alpha =\beta$.
\end{proof}

\noindent 
 We deduce a classification of additive cyclic orders on $( \zz; +)$:
\begin{prop} \label{Prop: Cyclicorderclassification}
Every additive cyclic order on $\zz$ is either $C_{+}$, $C_{-}$, or $C_\alpha$ for some $\alpha \in \rr \setminus \qq$. Moreover, for $\alpha, \beta \in \rr \setminus \qq$, $C_\alpha = C_\beta$  if and only if $\alpha -\beta \in \zz$. 
\end{prop}

\begin{proof}
Suppose $C$ is an additive cyclic order on $\zz$. It follows from Proposition \ref{prop: cyclicvslinear} and Lemma \ref{lem: linearorderclassification} that $(\zz; +, C)$ is isomorphic to either $(\zz; +, C_+)$ or $(\zz; +, C_\alpha)$ for $\alpha \in \rr \setminus \qq$. Note that the only group automorphism of $(\zz; +)$ are $\text{id}_\zz$ and $k \mapsto -k$. The latter maps $C_+$ to $C_{-}$ and $C_\alpha$ to $C_{-\alpha}$ for all $\alpha \in \rr \setminus \qq$. The first statement of the proposition follows.

The backward direction of the second statement follows from the easy observations that $C_\alpha = C_{\alpha+1}$. For the forward direction of the second statement, suppose $\alpha, \beta \in \rr \setminus \qq$ and  $C_\alpha = C_\beta$. In particular, this implies that
$$(\zz; +, C_{-\alpha}) \cong (\zz; +, C_\alpha) \cong (\zz; +, C_\beta) \cong  (\zz; +, C_{-\beta}).$$
By the backward direction of the second statement, we can arrange that $\alpha$ and $\beta$ are in $[-1\slash 2, 1\slash 2)_{\rr \setminus \qq}$. If  both $\alpha$ and $\beta$ are in  $[0, 1\slash 2)_{\rr \setminus \qq}$, then it follows from Lemma \ref{lem: linearorderclassification} that $\alpha =\beta$. If   both $\alpha$ and $\beta$ are in  $[-1\slash 2, 0)_{\rr \setminus \qq}$, a similar argument shows that $-\alpha = -\beta$, and so $\alpha =\beta$. Finally, suppose one out of $\alpha, \beta$ is in  $[-1\slash 2, 0)_{\rr \setminus \qq}$ and the other is in $[0, 1\slash 2)_{\rr \setminus \qq}$. A similar argument as the previous cases give us that $\alpha = -\beta$. However, $C_\alpha$ is always different from $C_{-\alpha}$, so this last case never happens.
 \end{proof}

\noindent We also need a well-known result of Kronecker: If $( \alpha_1, \ldots, \alpha_n) \in \rr^n$ is a $\qq$-linearly independent tuple of variables, then $$ \big( \alpha_1 m +\zz \ldots, \alpha_n m + \zz \big)_{ m >0} \text{ is dense in }  (\rr \slash \zz )^n,$$
where the latter is equipped with the obvious topology. See also \cite{Minh-1} for another instance where a phenomenon of this type is of central importance in dealing with cyclic orders.

\begin{thm}
Let  $\alpha $ be in $\rr \setminus \qq$. Then $( \zz; +, C_\alpha)$ is a reduct of neither $(\mathbb{Z};+,<)$ nor $(\mathbb{Z}; +, \prec_p)$ for any prime $p$.
\end{thm}

\begin{proof}
Suppose the notations are as given. We will show that 
$X =\{ k : C(0, k, 1)\}$ is definable neither in $(\mathbb{Z};+,<)$ nor $(\mathbb{Z}; +, \prec_p)$. By  \cite[Remark 3.2]{AldE}, any subset of $\zz$ definable in  $(\mathbb{Z}; +, \prec_p)$ is definable in $(\zz; +)$. Hence, it suffices to show that $X$ is not definable in $(\mathbb{Z};+,<)$.

Toward a contradiction, suppose $X$ is definable in $(\mathbb{Z};+,<)$. By the one-dimensional case of Kronecker's approximation theorem, we get that  both $X$ and $\zz \setminus X$ are infinite. It then follows easily from the quantifier elimination for $(\mathbb{Z};+,<)$ that there is $k\neq 0$ and $l$ such that  $$\{ km + l : m > 0\} \subseteq \zz \setminus X.$$
On the other hand, by Kronecker's approximation theorem again, we have that $X \cap \{ km + l : m > 0\} \neq \emptyset$ for all $k \neq 0$ and all $l$, which is absurd. 
\end{proof}

\section{Unary definable sets and definable equivalence}

\noindent We now show that if $\alpha,\beta \in \rr \setminus \qq$ are $\qq$-linearly independent then $(\zz;+,C_\alpha)$ does not define $C_\beta$.
This follows from a characterization of unary definable sets in a cyclically ordered expansion of $(\zz; +)$ and Kronecker's approximation theorem.

\medskip \noindent Let $C$ be a cyclic order on a set $G$. A subset $J$ of $G$ is {\bf convex} (with respect to $C$) if whenever $a,b \in J$ are distinct we either have $\{ t : C(a,t,b) \} \subseteq J$ or $\{ t : C(b,t,a) \} \subseteq J$.
Intervals are convex, and it is easy to see that the union of a nested family of convex sets is convex.

\begin{lem}\label{lem:convex}
 Let $(G;+,C)$ be densely cyclically ordered abelian group with universal cover $(H;u,+,<)$ and covering map $\pi : H \to G$.
If $J \subseteq H$ is convex (with respect to $<$) then $\pi(J)$ is convex (with respect to $C$).
\end{lem}

\begin{proof}
Let $J \subseteq H$ be convex.
Then $J$ is the union of a nested family of closed intervals $\{ I_a : a \in L\}$, i.e. we either have $I_a \subseteq I_b$ or $I_b \subseteq I_a$ for all $a,b \in  L$.
It follows that $\pi(J)$ is the union of the nested family $\{ \pi(I_a) : a \in L\}$.
It suffices to show that $\pi(J)$ is convex when $J$ is a closed interval.
Suppose $J = [g,h]$.

We first suppose $h - g \geq u$.
Then $[0,u]_H \subseteq J - g$.
The restriction of $\pi$ to $[0,u]_H$ is a surjection so $\pi(J - g) = G$.
As $\pi(J - g) = \pi(J) - \pi(g)$, we have $\pi(J) = G + \pi(g) = G$.
So in particular $\pi(J)$ is convex.
Now suppose $h - g < u$.
Then $J - g \subseteq [0,u]_H$.
It follows that $$\pi(J - g) = \{ t \in G : C(0,t, \pi(g - h) )\}$$ so $\pi(J - g)$ is convex.
Then $\pi(J) = \pi(J - g) + \pi(g)$ is a translate of a convex set and is hence convex.
\end{proof}

 \medskip \noindent  Suppose $(G; +, \ldots)$ be a structure where $(G; +)$ is an abelian group and $(G;  \ldots)$ expands either a linear order $<$ or a cyclic order $C$; convexity in the definitions below is with respect to either $<$ or $C$. A {\bf tmc-set}  is a translation of a multiple of a convex subset of $G$, that is, a subset of $G$ the form $a + mJ$ with $a \in G$ and convex $J\subseteq G$. A {\bf cnc-set} is a set of the form $J \cap (a+ nG)$ with  convex $J \subseteq G$ and $a \in G$.  
 
\medskip \noindent We say that $(G; +, \ldots)$  is {\bf tmc-minimal} if every definable unary set is a finite  union of tmc-sets and that $(G; +, \ldots)$  is {\bf cnc-minimal} if every definable unary set is a finite  union of cnc-sets. These two notions coincide for linearly ordered groups.

\begin{lem}\label{lem:basic-cnc}
Suppose that $(G;+,<)$ is a linearly ordered group.
Then the collection of tmc-sets and the collection of cnc-sets coincide.
\end{lem}

\begin{proof}
Let $X \subseteq G$ be an cnc-set.
Let $X = I \cap A$ for a convex $I \subseteq G$ and $A  = a + n G$.
Let $J = \{ g \in G : ng \in I - a \}$.
Monotonocity of $g \mapsto ng$ implies $J$ is convex as $I - a$ is convex.
The definition of $J$ implies $g \in J$ if and only if $a + ng \in I$.
As $a + ng \in A$ for all $g \in G$ we have $g \in J$ if and only if $a + ng \in I \cap A$.
So $X = a + nJ$.

Conversely, suppose $J$ is convex.
A translate of an cnc-set is an cnc-set, so it suffices to show $nJ$ is an cnc-set.
Let $I$ be the convex hull of $nJ$.
Then $nJ \subseteq I \cap nG$.
We show the other inclusion.
Suppose $g \in G$ and $ng \in I$.
Then $nh \leq ng \leq  nh'$ for some $nh,nh' \in nJ$.
Then $h \leq g \leq h'$, so $g \in  J$ as $h,h' \in J$ and $J$ is convex.
Thus $ng \in nJ$. 
\end{proof}

\noindent In cyclically ordered abelian groups there may be tmc-sets which are not cnc-sets. More precisely, it can be shown that there are tmc-sets which are not even finite unions of cnc-sets. An example is the set $\{ 2k : \alpha k \in [0, 1\slash 2) +\zz\}$ in the structure $( \zz; +, C_\alpha)$ with $\alpha \in \rr \setminus \qq$. As this will not be used later, we leave the proof to the interested readers.

\begin{lem}\label{lem:lin-cnc}
If $\alpha \in \rr \setminus \qq$ then $(\zz^2;+,<_\alpha)$ is cnc-minimal.
\end{lem}

\begin{proof}
The structure $(\zz^2;+,<_\alpha)$ admits quantifier elimination in the extended language where we add a predicate symbol defining $n\zz$ for each $n$. see  \cite{weis}, for example. 
It follows that any definable subset of $\zz^2$ is a finite union of finite intersections of sets of one of the following types:
\begin{enumerate}
\item $\{ t : k_1 t + a <_\alpha k_2 t + b\}$ for some $k_1,k_2$ and $a,b \in \zz^2$,
\item $\{ t : k_1 t + a \geq_\alpha k_2 t + b\}$ for some $k_1,k_2$ and $a,b \in \zz^2$,
\item $\{ t : k t + a \in n\zz^2\}$ for some $k,n$ and $a \in \zz^2$,
\item $\{ t : k t + a \notin n\zz^2\}$ for some $k,n$ and $a \in \zz^2$,
\item $\{ t : k_1 t + a_1 = k_2 t + a_2 \}$ for some $k_1,k_2$ and $a_1,a_2 \in \zz^2$,
\item $\{ t : k_1 t + a_1 \neq k_2 t + a_2\}$ for some $k_1,k_2$ and $a_1,a_2 \in \zz^2$.
\end{enumerate}
We show that any finite intersection of sets of type $(1)$-$(6)$ is a finite union of cnc-sets.
Every set of type $(1)$ or $(2)$ is either upwards or downwards closed.
It follows that any intersection of such sets is convex.

Suppose
$ A = \{ t : kt + a \in n\zz^2\}.$
Suppose $A$ is nonempty and $t' \in A$.
Then $kt + a \in n\zz^2$ if and only if 
$$ (kt + a) - (kt' + a) = k(t - t') \in n\zz^2.$$
For any $m$ we have $km \in n\zz$ if and only if $m$ is in $N\zz$ where $N = n/ \gcd(k,n)$.
So $t \in A$ if and only if $t - t' \in N\zz^2$, equivalently if $t \in N\zz^2 + t'$.
So $A$ is a coset of a subgroup of the form $N\zz^2$.
So any finite intersection of sets of type $(3)$ and $(4)$ is a boolean combination of cosets of subgroups of the form $n\zz^2$.
As $|\zz^2/n\zz^2| < \infty$, a complement of a coset of a subgroup of the form $n\zz^2$ is a finite union of such cosets.
It follows that any boolean combination of cosets of subgroups of the form $n\zz^2$ is a finite union of such cosets.

We have shown that a finite intersection of sets of type $(1)$-$(4)$ is an intersection of a convex set by a finite union of cosets of subgroups of the form $n\zz^2$.
It follows that any finite intersection of sets of type $(1)$-$(4)$ is a finite union of cnc-sets.

Any set of type $(5)$ or $(6)$ is either empty, $\zz^2$, a singleton, or the complement of a singleton.
It follows that any finite intersection of such sets is either finite or co-finite.
Suppose that $X$ is a finite union of cnc-sets.
The intersection of a $X$ and a finite set is finite, hence is a finite union of cnc-sets.
It is easy to see that the intersection of $X$ and a co-finite set is a finite union of cnc-sets.
\end{proof}

\begin{thm}\label{thm:cnc-minimality}
Let $\alpha \in \rr \setminus \qq$.
Then $(\zz;+,C_\alpha)$ is tmc-minimal.
\end{thm}

\begin{proof}

Suppose $X \subseteq \zz$ is definable.
Set $$Y = \pi^{-1}(X) \cap [0,u)_{\zz^2}.$$
Then $X = \pi(Y)$ and  $Y$ is a finite union of cnc-sets $Y_1,\ldots,Y_k$ by \ref{lem:lin-cnc}. As 
$$ \pi(Y) = \pi(Y_1) \cup \ldots \cup \pi(Y_k)$$
we may assume $Y$ is an cnc-set.
Applying Lemma~\ref{lem:basic-cnc} we suppose that $Y = a + nJ$ for $a \in \zz^2$ and convex $J \subseteq \zz^2$.
As $\pi$ is a homomorphism we have
$$ X = \pi(Y) = \pi(a) + n\pi(J).$$
It follows from Lemma~\ref{lem:convex} that $\pi(J)$ is convex.
Thus $X$ is a tmc-set.
\end{proof}

\noindent We say that $X \subseteq \zz$ is $C_\alpha$-dense if it is dense with respect to the obvious topology induced by $C_\alpha$.

\begin{lem}\label{lem:kron}
Suppose, $\alpha$ and $\beta$ in $\rr \setminus \qq$ are $\qq$-linearly independent and $J_\beta \subseteq \zz$ is $C_\beta$-convex and infinite, fix $n \geq 1,k$.
Then $k + nJ_\beta$ is $C_\alpha$-dense.
\end{lem}

\begin{proof}
Suppose $X \subseteq \zz$ is  $C_\alpha$-dense.
It follows by elementary topology that the image of $X$ under the map $l \mapsto k + nl$ is $C_\alpha$-dense in $k + n\zz$.
As $k + n\zz$ is $C_\alpha$-dense, it follows that $k + nX$ is $C_\alpha$-dense.
It therefore suffices to show that $J_\beta$ is dense with respect to the topology induced by $C_\alpha$.
We show that $J_\beta$ intersects an arbitrary infinite $C_\alpha$-convex $J_\alpha \subseteq \zz$.
Let $J'_\alpha$ and $J'_\beta$ be $C$-convex subsets of $\rr/\zz$ such that $J_\alpha = \chi_\alpha^{-1}(J'_\alpha)$ and $J_\beta = \chi^{-1}_\beta(J'_\beta)$.
Then $J'_\alpha, J'_\beta$ are infinite and so have nonempty interior.
It follows from $\qq$-linear independence of $\alpha$ and $\beta$ and Kronecker's theorem that 
$$\{   (\chi_\alpha(m), \chi_\beta(m)) :  m \in \zz \} \ \text{ is dense in }\ (\rr / \zz)^2. $$  In particular, there is $m \in \zz$ such that  $\big(\chi_\alpha(m), \chi_\beta(m)\big)  \in J'_\alpha \times J'_\beta$. Then $m$ is in $J_\alpha\cap J_\beta$, which implies that the latter is non-empty.
\end{proof}

\begin{cor} \label{cor: linearindep}
Suppose $\alpha, \beta \in \rr \setminus \qq$ are $\qq$-linearly independent. Then there is a $(\zz;+,C_\beta)$-definable subset of $\zz$ which is not definable in $( \zz; +, C_\alpha)$.
\end{cor}
\begin{proof}
Suppose that $\alpha$ and $\beta$ are $\mathbb{Q}$-linearly independent elements of $\rr\setminus \qq$.
Let $J_\alpha$ be an infinite $C_\alpha$-convex set definable in $(\mathbb{Z};+,C_\alpha)$ with infinite complement. Suppose $\zz \setminus J_\alpha$ is definable in $(\mathbb{Z};+,C_\beta)$. It follows from tmc-minmality of the latter that $\zz \setminus J_\alpha \supseteq k+ nJ_\beta$ where $J_\beta$ is $C_\beta$-covex and $n \geq 1$.
Lemma~\ref{lem:kron} shows that $k + nJ_\beta$ is $C_\alpha$-dense and thus intersects $J_\alpha$, contradiction.
\end{proof}

\noindent As consequence of Corollary \ref{cor: linearindep} we obtain uncountably many definably distinct dp-minimal expansions of $(\zz;+)$.

\begin{cor}
There are continumn many pairwise definably distinct cyclically ordered groups expanding $(\zz;+)$.
\end{cor}

\noindent We now show that if $\alpha,\beta \in \rr \setminus \qq$ are $\qq$-linearly dependent then $C_\beta$ is $(\zz;+,C_\alpha)$-definable.
It follows that $(\zz;+,C_\alpha)$ and $(\zz;+,C_\beta)$ are definably equivalent if and only if $\alpha$ and $\beta$ are $\qq$-linearly dependent, i,e, if $\beta = q\alpha + r$ for some $q,r \in \qq$.
This requires several steps.

\begin{lem} \label{lem:definabilityinterval}
Suppose $\alpha$ is in $\rr \setminus \qq$, $n$ is in $\nn^{\geq 1}$, and $r$ is in $\{0, \ldots, n-1\}$. Then the set $$  \big\{ l \ :\ \alpha l +\zz \in [r/n, (r+1)/n) +\zz \big\}$$ is definable in $( \zz; +, C_\alpha)$.
\end{lem}
\begin{proof}
Let the notation be as given and $( \rr \slash \zz; +, C)$ be the oriented circle. We have that $\alpha l + \zz $ is in $ [r/n, (r+1)/n) + \zz$ if and only if $( \alpha il + \zz)_{i=0}^{n}$ ``winds'' $r$ times around $\rr \slash \zz$, that is, 
$$C\big( 0 + \zz, \alpha (i+1) l + \zz, \alpha i l + \zz \big)  \text{ holds for exactly }  r \text{ values of } i\in \{1, \ldots, n-1  \}.$$ 
The desired conclusion follows.
\end{proof}

\begin{cor}\label{cor: lineardep}
If $\alpha$ and $\beta$ are in $\rr \setminus \qq$ and $\beta =\alpha + m\slash n$ with $n \geq 1$, then $C_\beta$ is definable in $(\mathbb{Z};+,C_\alpha)$.
\end{cor}
\begin{proof}
Suppose $\alpha$ is in $\rr \setminus \qq$. Note that $C_{-\alpha}(j,k,l)$ if and only if $C_{\alpha}(-j,-k,-l)$, so $C_{-\alpha}$ is definable in $(\zz;+,C_\alpha)$.
As $\alpha - m \slash n = - (-\alpha + m\slash n)$ is suffices to treat the case when $m \geq 1$.
It suffices to treat the case $\beta = \alpha + 1 \slash n$ and then apply this case $m$ times to get the general case.

Suppose $\alpha, \beta$ are in $\rr \setminus \qq$ and $\beta =\alpha + 1/n$ with $n \geq 1$. As $C_\alpha$ is additive it suffices to show that the set of pairs $(k, l)$ such that $C_\beta( 0, k, l)$ is definable in $(\mathbb{Z};+,C_\alpha)$. Let $( \rr \slash \zz; +, C)$ be the positively oriented circle.  By definition,  $C_\beta( 0, k, l)$ is equivalent to  $C( 0+ \zz, \beta k +\zz , \beta l + \zz)$. The latter holds if and only if either there are $r, s \in \{0, \ldots, n-1\} \text{ with }r<s$ such that
$$ \beta k + \zz \in [r/n, (r+1)/n)+\zz \  \text{ and } \  \beta l +\zz \in [s/n,(s+1)/n) +\zz $$  or there is $r \in \{0, \ldots, n-1\}$ such that
$$\beta k +\zz, \beta l +\zz \in [r/n, (r+1)/n) + \zz  \text{ and } C\big(0, \beta nk +\zz, \beta n l +\zz \big) .$$ 
For all $a \in \rr$, we have that  $a + k \slash n + \zz \in [r/n,  (r+1)/n)+ \zz $ holds if an only if $a $ is in $[r'/n,  (r'+1)/n)+ \zz$ with  $r' \in \{0, \ldots, n-1\} \text{ and  } r'+k \equiv r \pmod n$. Hence, it follows from $\beta = \alpha+1\slash n$ that $\beta k  + \zz \in [r/n,  (r+1)/n) + \zz $ is equivalent to  $$ \alpha k + \zz \in [r'/n, (r'+1)/n) + \zz \ \text{ with } r' \in \{0, \ldots, n-1\} \text{ and  } r'+k \equiv r \pmod n.$$ On the other hand, as $n\beta =n\alpha +1$, so we get  $$C(0 + \zz, \beta nk +\zz, \beta nl +\zz) \  \text{ is equivalent to }\ C(0 + \zz, \alpha nk +\zz, \alpha nl +\zz).$$ By definition of $C_\alpha$, the latter holds if and only if $C_{\alpha}(0, nk, nl)$. Combining with Lemma \ref{lem:definabilityinterval} we get the desired conclusion.
\end{proof}

\begin{lem} \label{lem:definabilityinterval2} 
Suppose $\alpha$ is in $[0, 1)_{\rr \setminus \qq}$, $m,n$ are in $\nn^{\geq 1}$, and $r$ is in $\{0, \ldots, n-1\}$. Then the set $$  \big\{ l \ :\  \alpha l +\zz \in [0, r\alpha\slash n) +\zz \big\}$$ is definable in $( \zz; +, C_\alpha)$.
\end{lem}
\begin{proof}
 Suppose $\alpha$,$n$, and $r$ are as given and $( \rr \slash \zz; +, C)$ is the positively oriented circle. We note that $\alpha l + \zz$ is in $[0, \alpha\slash n) +\zz$ if and only if $C( 0 + \zz, \alpha nl + \zz , \alpha + \zz)$ and $( \alpha il + \zz)_{i=0}^{n}$ does not``winds''  around $\rr \slash \zz$, that is, $$C( 0 +\zz, \alpha il +\zz , \alpha (i+1)l +\zz) \text{ for all } i \in \{1, \ldots, n-1\}.$$ Recall that by definition $C(\alpha j +\zz, \alpha k+\zz, \alpha l+\zz)$  if and only if  $C_\alpha(j, k, l)$. Hence, $$ \big\{ l \ :\  \alpha l +\zz \in [0, \alpha\slash n) +\zz \big\} \text{ is definable in } ( \zz; +, C_\alpha).$$ The conclusion follow the easy observation that $\alpha l +\zz$ is in $ [0, r\alpha\slash n)+\zz $ if and only if $C_\alpha( 0, l, rk)$ for some $ k \in [0, \alpha\slash n) + \zz$. 
\end{proof}

\begin{cor}\label{cor:lineardep2}
Suppose $\alpha$ is in $[0, 1)_{\rr \setminus \qq}$, $n$ is in $\nn^{\geq 1}$, and $\beta = m\alpha \slash n$. Then $C_\beta$ is definable in  $( \zz; +, C_\alpha)$.
\end{cor}
\begin{proof}
As $\chi_{\alpha\slash n}(m k) = \chi_{m\alpha\slash n}(k)$ for all $k$ we have $C_{m\alpha/n}(i,j,l)$ if and only if $C_{\alpha \slash n}(mi,mj,ml)$.
It therefore suffices to treat the case $\beta = \alpha\slash n$.
For any given $k$ and $r \in \{0, 1, \ldots, n\}$, let $$ X_{k,r} = \{ l :  \alpha l +\zz \in [0, k\alpha + r\alpha\slash n) + \zz \}.$$ 
We first prove that $X_{k,r}$ is definable in $( \zz; +, C_\alpha)$ for all $k$ and $r$ as above. This is true for $r=0$ as $l \in X_{k,0}$ if and only if either $l =0$ or $C(0 +\zz, l \alpha +\zz, k \alpha +\zz)$. The later is equivalent to $C_\alpha( 0, k, l)$ by definition. The case where $k =0$ is just the preceding lemma. In general, we have that 
$$X_{k,r} = 
\begin{cases}
X_{k, 0} \cup (k+X_{0, r} )& \text{ if } X_{k, 0} \cap  (k+X_{0, r} ) =\emptyset ,\\
X_{k, 0} \cap  (k+X_{0, r} ) & \text{ otherwise}.
\end{cases}$$ 
Let   $r, s $ be in $\{0, \ldots, n-1\}$. We have that $C_\beta( 0, kn+ r, ln+s)$ is equivalent to $C( 0 +\zz; \beta(kn+r) + \zz, \beta(ln+s) +\zz)$ by definition. The latter holds if and only if $kn+ r$, $ln+ s$, and $0$ are all distinct and $ X_{k,r} \subseteq X_{l,r}$. The conclusion follows.
\end{proof}

\noindent Corollary~\ref{cor: lineardep} and Corollary~\ref{cor:lineardep2} show that $C_\beta$ is definable in $(\zz;+,C_\alpha)$ whenever $\alpha,\beta \in \rr \setminus \qq$ are $\qq$-linearly dependent.
Combining with Corollary~\ref{cor: linearindep} we get:

\begin{thm}\label{thm:diff}
Suppose $\alpha$ and $\beta$ are in $\rr\setminus \qq$.
Then $(\zz;+,C_\alpha)$ and $(\zz;+,C_\beta)$ are definably equivalent if and only if $\alpha,\beta$ are $\qq$-linearly dependent.
\end{thm}

\noindent  Finally , we give an example of a dp-minimal expansion of $(\zz;+)$ which defines uncountably many subsets of $\zz$.
Let $\mathscr{M} = (M,\ldots)$ be a structure and $\mathscr{N} = (N,\ldots)$ be a highly saturated elementary expansion of $\mathscr{M}$.
Then a subset of $M^k$ is \textbf{externally definable} if it is of the form $A \cap M^k \text{ where } A \subseteq N^k \text{ is definable in }\mathscr{N}. $
A standard saturation argument shows that the collection of externally definable sets does not depend on the choice of $\mathscr{N}$.
The \textbf{Shelah expansion} of $\mathscr{M}$ is the expansion $\mathscr{M}^{\mathrm{Sh}}$ of $\mathscr{M}$ obtained by adding a predicate defining every externally definable subset of every $M^k$.
It was shown in \cite{Shelah} that $\mathscr{M}^{\mathrm{Sh}}$ is NIP whenever $\mathscr{M}$ is, see also \cite[Chapter 3]{simon-book}.
It was observed in \cite[3.8]{OnUs} that the main theorem of \cite{Shelah} also shows that $\mathscr{M}^\mathrm{Sh}$ is dp-minimal whenever $\mathscr{M}$ is dp-minimal.
In particular $(\zz;+,C_\alpha)^\mathrm{Sh}$ is dp-minimal for any $\alpha \in \rr \setminus \qq$.

\begin{prop}
Fix $\alpha \in \rr \setminus \qq$.
Then $(\zz;+,C_\alpha)^\mathrm{Sh}$ defines uncountably many distinct subsets of $\zz$ and has uncountably many definably distinct reducts.
\end{prop}

\begin{proof}
If $\mathscr{M},\mathscr{N}$ are as above, and $\mathscr{M}$ expands a linear or cyclic order then it is easy to see that any convex subset of $M$ is of the form $I \cap M$ for an interval $I \subseteq N$.
It follows that $(\zz;+,C_\alpha)^\mathrm{Sh}$ defines every $C_\alpha$-convex subset of $\zz$ and thus defines uncountably many subsets of $\zz$.
Any reduct of $(\zz;+,C_\alpha)^\mathrm{Sh}$ to a countable language defines only countably many subsets of $\zz$, it follows that $(\zz;+,C_\alpha)^\mathrm{Sh}$ has uncountably many definably distinct reducts.
\end{proof}

\section{Further questions}
\noindent There are several facts we  would like to know about cyclic group orders on $(\zz; +)$. If $\alpha$ and $\beta$ are linearly independent elements of $\rr \setminus \qq$, is $(\mathbb{Z};+,C_\alpha, C_\beta)$ $\text{NIP}$? Is there any reduct of $(\mathbb{Z};+,C_\alpha)$ with $\alpha \in \rr\setminus \qq$ which is definably distinct from $(\mathbb{Z};+,C_{\alpha})$ and $(\zz;+)$?
Is there a $p$-adic analogue of $(\zz;+, C_\alpha)$?
Finally, may the classification question for dp-minimal expansions $(\mathbb{Z};+)$ be recovered in any form?

\section*{Acknowledgements}
\noindent The authors thank the organizers of midwest model theory day, at which this research was begun.
They also thank Nigel Pynn-Coates for providing a good work environment on the drive back.
The authors acknowledge support from NSF grant DMS-1654725.

\bibliographystyle{amsalpha}
\bibliography{the}

\newcommand{\etalchar}[1]{$^{#1}$}
\providecommand{\bysame}{\leavevmode\hbox to3em{\hrulefill}\thinspace}
\providecommand{\MR}{\relax\ifhmode\unskip\space\fi MR }
\providecommand{\MRhref}[2]{%
  \href{http://www.ams.org/mathscinet-getitem?mr=#1}{#2}
}
\providecommand{\href}[2]{#2}
\begin{thebibliography}{ADH{\etalchar{+}}16}

\bibitem[Ad17]{AldE}
Eran Alouf and Christian d'Elb\'ee, \emph{A new minimal expansion of the
  integers}, arXiv:1707.07203 (2017).

\bibitem[ADH{\etalchar{+}}13]{toomanyII}
Matthias Aschenbrenner, Alf Dolich, Deirdre Haskell, Dugald Macpherson, and
  Sergei Starchenko, \emph{Vapnik-{C}hervonenkis density in some theories
  without the independence property, {II}}, Notre Dame J. Form. Log.
  \textbf{54} (2013), no.~3-4, 311--363. \MR{3091661}

\bibitem[ADH{\etalchar{+}}16]{toomanyI}
\bysame, \emph{Vapnik-{C}hervonenkis density in some theories without the
  independence property, {I}}, Trans. Amer. Math. Soc. \textbf{368} (2016),
  no.~8, 5889--5949. \MR{3458402}

\bibitem[Con16]{conant}
Gabriel Conant, \emph{There are no intermediate structures between the group of
  integers and presburger arithmetic}, Journal of Symbolic Logic \textbf{to
  appear} (2016).

\bibitem[CP16]{CoPi}
Gabriel Conant and Anand Pillay, \emph{Stable groups and expansions of
  $(\mathbb{Z},+,0)$}, Fundamenta Mathematicae \textbf{to appear} (2016).

\bibitem[DG17]{DoGo}
Alfred Dolich and John Goodrick, \emph{Strong theories of ordered {A}belian
  groups}, Fund. Math. \textbf{236} (2017), no.~3, 269--296. \MR{3600762}

\bibitem[Gun08]{Gu-thesis}
Ayhan Gunaydin, \emph{Model theory of fields with multiplicative groups},
  ProQuest LLC, Ann Arbor, MI, 2008, Thesis (Ph.D.)--University of Illinois at
  Urbana-Champaign. \MR{2712584}

\bibitem[Joh15]{Johnson}
Will Johnson, \emph{On dp-minimal fields}, arXiv:1507.02745 (2015).

\bibitem[JSW17]{JaSiWa}
Franziska Jahnke, Pierre Simon, and Erik Walsberg, \emph{Dp-minimal valued
  fields}, J. Symb. Log. \textbf{82} (2017), no.~1, 151--165. \MR{3631280}

\bibitem[OU11]{OnUs}
Alf Onshuus and Alexander Usvyatsov, \emph{On dp-minimality, strong dependence
  and weight}, J. Symbolic Logic \textbf{76} (2011), no.~3, 737--758.
  \MR{2849244}

\bibitem[She09]{Shelah}
Saharon Shelah, \emph{Dependent first order theories, continued}, Israel J.
  Math. \textbf{173} (2009), 1--60. \MR{2570659}

\bibitem[Sim15]{simon-book}
Pierre Simon, \emph{A guide to {NIP} theories}, Lecture Notes in Logic,
  vol.~44, Association for Symbolic Logic, Chicago, IL; Cambridge Scientific
  Publishers, Cambridge, 2015. \MR{3560428}

\bibitem[Tra17]{Minh-1}
Minh~Chieu Tran, \emph{{Tame structures via multiplicative character sums on
  varieties over finite fields}}, ArXiv e-prints (2017).

\bibitem[Wei81]{weis}
Volker Weispfenning, \emph{Elimination of quantifiers for certain ordered and
  lattice-ordered abelian groups}, Bull. Soc. Math. Belg. S\'er. B \textbf{33}
  (1981), no.~1, 131--155. \MR{620968}

\end{thebibliography}

\end{document}